\documentclass[12pt]{amsart}
\usepackage{amsmath}
\usepackage{amssymb,amscd}
\usepackage[T1]{fontenc}
\usepackage[latin1]{inputenc}
\usepackage[english]{babel}
\usepackage{graphicx}

\numberwithin{equation}{section}

\newtheorem{theorem}{Theorem}[section]

\newtheorem{corollary}[theorem]{Corollary}

\newtheorem{lemma}[theorem]{Lemma}
\newtheorem{proposition}[theorem]{Proposition}
\newtheorem{remark}[theorem]{Remark}

\def \begineq{\begin{equation}}
\def \endeq{\end{equation}}

\def \({\left(}
\def \){\right)}
\def \<{\langle}
\def \>{\rangle}
\def \bar{\overline}

\begin{document}
 \thispagestyle{plain}

\title[Geometric Quantization of Toda systems]{Geometric quantization of finite Toda systems and coherent states}
\author[Rukmini Dey and Saibal Ganguli]{Rukmini Dey and Saibal Ganguli}

\date{4/3/2017}

\maketitle
\begin{abstract}
 Adler had showed that the Toda system can be given a coadjoint orbit description. We quantize the Toda system by viewing it as a single orbit of a multiplicative group of lower
 triangular matrices of determinant one with positive
diagonal entries. We get a unitary
representation of the group with square integrable polarized sections of the quantization as the module . We find the Rawnsley coherent
states after completion of the above space of sections. We also find non-unitary finite dimensional quantum Hilbert spaces for the system. Finally we give an expression for the quantum Hamiltonian for the system.\\[0.2cm]
\textsl{MSC}:Differential geometry 53xx, Quantum Theory 81xx.\\
\textsl{Keywords}:geometric quantization, coadjoint orbit,induced representation,Toda.
\end{abstract} 

\label{first}

 \section{ Introduction}
 The connection between finite Toda system and coadjoint orbits was first explored by Adler ~\cite{A}.
 We summarize the introduction to the Toda system as in ~\cite{A}.
 The Hamiltonian  considered is  $H = \frac{1}{2} \sum_{i=1}^{n} y_i^2 + \sum_{i=1}^{n} {e}^{x_i - x_{i+1}}$,  $x_0 = x_{n+1}$.  The Hamiltonian equations  
 are  $\stackrel{\cdot}{{x_i}}= y_i,$ $\stackrel{\cdot}{{y_i}} = {}^{x_{i-1} - x_i} - {e}^{x_i - x_{i+1}}$, $i=1,...,n$.

Define 

\begin{center}
$a_i = \frac{1}{2} {e}^{\frac{1}{2} (x_i - x_{i +1})} $, $i = 1,...,n-1$, \\

$b_i = \frac{1}{2} y_i ,$ $i=1,...,n$.
\end{center}

Note that $a_i  >  0$ for $i =1,...,n-1$.

Adler showed that the Hamiltonian equation of motion corresponds to a Lax equation and gave explicit expression for  the integrals of motion which Poisson commute w.r.t. the following Poisson bracket.

\begin{center}
$\{ f, g \} = \stackrel{\cdot}{\sum}( a_{i-1} g_{a_{i-1}} - a_i g_{a_{i}} ) f_{b_{i}} 
+ \stackrel{\cdot}{\sum} a_i (g_{b_{i}} - g_{b_{i+1}}) f_{a_{i}}$ 
\end{center}
where $\cdot$ means to omit terms with undefined elements, i.e. terms involving 
$a_0, a_n, b_{n+1}$. 

({\bf Note:} In page 222, ~\cite{A}, there is a misprint in the formula for the Poisson bracket. The correct formula for the Poisson bracket, as written above, is given on page 225 in the same paper).

Adler goes on to show that the same system has a coadjoint orbit description of the group of lower triangular matrices of non-zero diagonal. In fact, one can restrict the  action to that of lower triangular triangular matrices of determinant $1$ and positive diagonal elements. The orbit is homeomorphic to 
${\mathbb{R}_+}^{n-1} \times {\mathbb{R}}^{n-1}$, 
just described by $a_i
>0$, $i = 1,..., n-1$ and $ b_i$, $i =1,...,n$  such that $b_1+ b_2 + ...+ b_n =c$, $c$ a constant. 

We describe this precisely and show that the Toda system corresponds to a single orbit. We explore it further to geometrically quantize the Toda system. This is possible since the orbit  has a symplectic structure whose induced Poisson bracket corresponds exactly to the Poisson bracket mentioned above.

In two famous papers ~\cite{K1}, ~\cite{K2}, Kostant describes the  Kostant-Souriau quantization in general and for coadjoint orbits 
in particular. Using his construction, we were able to construct an infinite dimensional Hilbert-space of 
polarized sections of the quantum bundle (which is trivial in our case). The polarized sections are square 
integrable functions of $a_i$, $i =1,..., (n-1)$ only.  In this construction  we  modified  the 
usual volume (given by the symplectic form) by an exponential decay.  The group of 
lower triangular matrices with determinant $1$ and positive diagonal entries acts on this Hilbert space giving in fact a unitary representation. By 
Kostant's result on 
general coadjoint orbits, ~\cite{K1}, ~\cite{K2},  this exactly corresponds to the Hilbert space of geometric quantization of the orbit and hence the Toda system.

Next, we construct Rawnsley coherent states, ~\cite{R}, of the Toda system corresponding to this quantization.
There is another definition of coherent states for orbits. The coherent states are obtained by moving any ``vacuum'' vector by the group action. We show  for orbit quantization, these two notions of coherent states coincide. 

We  also construct  (in the last section) finite dimensional representations which are not unitary.

The general reference we have followed of geometric quantization is the book by Woodhouse ~\cite{W}.

This paper is mainly an exercise in mathematics.
It would be interesting to relate this quantization with other quantizations of the Toda system which are relevant to physics, ~\cite{S}, ~\cite{SS}.

The  quantization of  Toda systems using geometric methods has already been worked out  by Reyman and Semenov-Tian-Shansky, ~\cite{RSe}, ~\cite{Se}  and Kharchev and Lebedev ~\cite{KL}. The  coherent states have not  been considered in this context as far as the authors know. 

It would be interesting to see if the same idea of geometric  quantization goes through for KdV type systems. 
The symplectic structure in this case has been derived by Adler,  ~\cite{A}.

 \section{ The Orbit Corresponding to  Toda System}\label{Toda_1}
  A Toda system is characterized by a upper triangular matrix  of the following form, ~\cite{A}, 

  \begin{equation}\label{mat}
  A=\begin{bmatrix}
     b_1 & a_1 &      \\ 
           & b_2 & a_2 &    \\
          &     & \ddots & \ddots\\
          &      &        & b_{n-1} & a_{n-1} \\
          &       &        &          & b_{n} \\
      \end{bmatrix}
      \end{equation}
         
       where the trace is constant and each $a_i$ is positive a and all other matrix elements are zero.
       
       The space of upper triangular matrix is the dual lie algebra of multiplicative of group of lower triangular matrix.  There is an coadjoint
       action. Since Toda matrices are upper triangular we  attempt to find its orbit structure. The orbit structure is described in the following 
       proposition.
       \begin{proposition}\label{orb}
        The Toda space described above is a codajoint orbit of the multiplicative group of lower triangular matrix with positive diagonals.
       \end{proposition}
\begin{proof}
       In general 
       the coadjoint action of $l$
       invertible and
       lower triangular on an upper triangular matrix $u$  described in  \cite{A} is given by.
       
       \begin{equation}
       l(u)={[l u l^{-1}]}_{+}
         \end{equation}
         
         where $[ ]_{+}$ mean projection to the  upper triangular part.
         
         We find the orbits of the  following matrices $C_{ii+1}$  where only ${(C_{ii+1})}_{ii+1} > 0$ and 
         rest of the terms
         are  zero. Let $L$ be any invertible lower triangular  matrix given by $L=(L_{ij})$
         then the matrix ${LC_{12}L^{-1}}_+$ is given by
         \begin{center}
         \begin{equation}\label{mat_1}
          ({[LC_{12}L^{-1}}]_+)_{11}=-\frac{L_{21}}{L_{22}} c,
         \end{equation}
          \begin{equation}\label{mat_2}
           ({[LC_{12}L^{-1}]}_+)_{12}=\frac{L_{11}}{L_{22}} c,
          \end{equation}
            \begin{equation}\label{mat_3}
           ({[LC_{12}L^{-1}]}_+)_{22}=\frac{L_{21}}{L_{22}} c
          \end{equation}
          \end{center}
          
         where $c$ is the only nonzero entry of $C_{12}$ that is ${(C_{12})}_{12}=c$.All other terms of the above matrix is zero.
         
        Now If $\frac{L_{21}}{L_{22}}$ and $\frac{L_{11}}{L_{22}}$ are allowed to vary freely we will get  a two dimensional  manifold.
           Proceeding further  by induction we can prove that the orbit of $C$ with diagonal entries zero , all $(i,i+1) $ entry 
           greater than zero and rest entries zero is 
         $2(n-1)$ dimensional.The inductive argument is presented in the following lemma \ref{ind}.
          If the  Toda matrix is  $T$ then
          
          \begin{equation}\label{toda}
          T= D + C
          \end{equation}

           where $D$ is diagonal and $C$ is a matrix of  the type described above. The action of the invertible
           lower triangular matrices with 
           positive diagonals will fix  $D$ and take the $C$ to  a $2(n-1)$ dimensional  manifold. Since the orbit map 
           from $G$ to the Toda space
           is of constant rank then by constant rank theorem the image (the orbit) is a manifold of
           dimension equal to the  constant rank of the
           differential of the orbit map. Since each image is
         a   $2(n-1)$ dimensional manifold, the rank  of the differential is  $2(n-1)$.  Since the Toda space is
          $2(n-1)$ dimensional the differentials are surjections. So by submersion theorem local open sets of 
          $G$  map surjectively by the orbit map
          to open sets of Toda space. Thus the orbits are open.
          
            Since by above orbits are open each orbit is open and closed. This is because if we take an orbit $O$, 
            it is open by the above argument and its
            complement is open as it is the union of other orbits.  We have  the set $O$ is open and   
            so is its complement,  making it an open and closed set. 
           
            Since the Toda system is a connected set there will be only one orbit corresponding to it. 
           \end{proof}

           \begin{lemma}\label{ind}
         
          The coadjoint orbit of the matrix of type $C$ is of dimension $2(n-1)$ 
            \end{lemma}
              \begin{proof}
             We proceed by induction.
            We first  try to understand the case $C_{12} +C_{23}$
            \begin{equation}\label{mat_4}
            ({[L (C_{12} + C_{23})L^{-1}]}_{+})_{11}=-\frac{L_{21}}{L_{22}} c_1,
            \end{equation}
               \begin{equation}\label{mat_5}
            ({[L (C_{12} + C_{23})L^{-1}]}_{+})_{12}=\frac{L_{11}}{L_{22}} c_1,
            \end{equation}
           \begin{equation}\label{mat_6}
            ({[L (C_{12} + C_{23})L^{-1}]}_{+})_{22}=\frac{L_{21}}{L_{22}} c_1 -\frac{L_{32}}{L_{33}}c_2,
            \end{equation}
          \begin{equation}\label{mat_7}
            ({[L (C_{12} + C_{23})L^{-1}]}_{+})_{23}=\frac{L_{22}}{L_{33}} c_2,
            \end{equation}
           \begin{equation}\label{mat_8}
            ({[L (C_{12} + C_{23})L^{-1}]}_{+})_{33}=\frac{L_{32}}{L_{33}} c_2.
            \end{equation}
        Here $c_1$ and $c_2$ are the values of  non zero entries of $C_{12}$ and $C_{23}$ that is ${(C_{12})}_{12}=c_1$ and $(C_{23})_{23}=c_2$ .All other terms of the 
        above matrix is zero.  
        Now for  fixing a  value of 
        $\frac{L_{32}}{L_{33}}c_2$
        the first row entries are unaffected so restriction of the orbit for  fixed $\frac{L_{32}}{L_{33}}c_2$  has 
        at least two dimensions. Now as we change
        $\frac{L_{32}}{L_{33}}c_2$ we get one more dimension thus the orbit will be 
        at least three dimensional and since orbits are symplectic and hence even dimensional, 
        this orbit will be at least  four dimensional. 
          
        Now  assuming the induction hypothesis that orbit  $C_{12} + \ldots + C_{k k+1} $ is at  least $2k$
        dimensional. Since the orbit of $C_{k+1 k+2}$ (which is
        similar to $C_{12}$ case discussed above) is two dimensional, 
        we can show $C_{12} + \ldots + C_{k k+1} +C_{k+1 k+2}$ will have orbit 
        at least 
        $2(k+1)$ dimensional (following the same line of argument as in the $4$-d case.)
          
           So the orbit of the matrix $C$ in the Toda equation \eqref{toda} above 
           is at least $2(n-1)$ dimensional and since it lies in the  Toda 
           system  which is 
           $2(n-1)$ dimensional and the  orbit of $C$  is exactly $2(n-1)$ dimensional.
      \end{proof}
 
 \section{The Symplectic Form}\label{symplectic}

 We summarize the description of the symplectic form if the coajoint orbit as in ~\cite{A}.  We also provide a formula for it.  It is important to note that it corresponds to the  Poisson bracket of the Toda system (given in page 225, ~\cite{A}). This has been shown in the same paper.
 
    Let G be the group of lower triangular matrices with non-zero diagonal entries.
Its Lie algebra L is the lower triangular matrices. 
 
As in ~\cite{A}, we may identify the dual of L, namely $L^{*}$ , with the upper triangular matrices, using the pairing $<A, B>=Tr(AB)$.
$g \in G$ acts on $L$ via conjugation. By duality, it acts on $l^{*} \in L^{*}$  $g : l^{*} \rightarrow
[{g}^{-1}  l^{*} g]_+$ , where $ [, ]_+ $ denotes the projection operator of setting all terms below the
diagonal equal to zero, ~\cite{A}.
This is the coadjoint representation of G. The orbit of this action is through
  $l^{*} \in L^{*}$ is $ \theta_{l^{*}} = \{[g^{-1} l^{*} g]_+ |g \in G\}$.
The tangent space of $\theta_{l^{*}}$ at $l^{*}$ is described by $T_{l^{*}} \theta_{l^{*}} = \{[l^{*} , l]_+ |l \in L\}$.
The Kostant-Kirillov 2-form $\omega$ associated with the orbit space $\theta_{l^{*}}$ is
\begin{center}
$\omega([l^{*} , l_1 ] + , [l^{*} , l_2 ]_+ )(l^{*} ) =< l^{*} , [l_1 , l_2 ] >=< [l^{*} ,l_1 ]_+ , l_2 >$.
\end{center}
When one writes $l_1^{*} = [l^{*} , l_1 ]_+$ and $l_2^{*} = [l^{*} , l_2 ]_+$ and solves for the symplectic form, one gets that the symplectic form is:
\begin{center}
$\omega= \sum _{i=1}^{n-1} \frac{1}{a_i} d(a_i) \wedge d(\sum_{j=1}^{i}b_{j})$.
\end{center}

This form is of integral cohomology class, because it is exact. 

Also, the symplectic form can be written in a standard form by a change of coordinates.
 Recall $a_i >0$ for $i=1,..., n-1$. 
Let $q_i = log (a_i)$ and $p_i = \sum_{j=1}^i b_j$. Then the symplectic form is  the standard symplectic form  on ${\mathbb{R}}^{2(n-1)}$, namely, 
\begin{center}
$\sum_{i=1}^{n-1} d(q_i) \wedge d(p_i).$ 
\end{center}

\section{Polarization of the Complexified Lie Algebra}
Here we descibe the notion of polarization in a Lie algebra as in Kostant ~\cite{K2}.
A polarization at $g$ of an element of the dual lie algebra of a group $G$ is a complex subalgebra  $\mathfrak {h} \subset \mathfrak{g}_{\mathbb{C}}$
(the complexification
 of $\mathfrak{g}$) which satisfies the following conditions 
 
 (1) the subalgebra is is stable under Ad $G_g$ ,
 
(2) $dim_{\mathbb{C}} (\frac{\mathfrak{g}_{\mathbb{C}}}{\mathfrak{h}})=n$,

(3) $< g , [ \mathfrak{h},\mathfrak{h} ] > = 0 $ and 

(4)$\mathfrak{h} + \bar{\mathfrak{h}} $  is a subalgebra of 
$\mathfrak{g}_{\mathbb{C}} $. 

  If we restrict our self  to the group $ G_1$  of lower triangular matrices with non-negative diagonal entries and determinant 1 
  then its lie-algebra $\mathfrak{g}$ is the 
  lower
  triangular matrices with trace zero. 
 \begin{proposition}\label{polar}
   The lie subalgebra of $\mathfrak{g}_{\mathbb{C}}$ , $\mathfrak{h}$ of complex 
  lower triangular matrices with diagonal zero is a polarization at any $g$ in Toda space.
\end{proposition}
  This involves trivial checking.
  
 \subsection{Quantization} Here we recall the notion of geometric quantziation as in Woodhouse ~\cite{W}. A classical system is given by a pair $(M,\omega)$ where $M$ is a symplectic manifold and $\omega$ a symplectic form. 
    Suppose  that the cohomology
class of $\omega $  is integral then geometric pre-quantization is the construction of a line bundle
$L$ over $M$ whose curvature $\rho(\nabla)$ (of a certain connection of $L$ with covariant derivative  $\nabla$) ,
 is proportional to $\omega$. The connection is given by a symplectic potential i.e. a one form $\theta$ such that $\omega = d(\theta)$ locally. The construction of the line bundle is  always possible as long as the cohomology class of $\omega$
is 
proportional to
an integral class. The Hilbert space of prequantization is  the completion of the space of square integrable sections of $L$. The Hilbert space of prequantization 
described above is too large for most purposes. Geometric quantization involves  construction of a polarization (a certain type of distribution ) of the symplectic manifold
such that we now take polarized sections (depending on the above distribution) of the 
line bundle, yielding a finite dimensional Hilbert space in most cases. 
   
  {\bf Polarizations on the Manifold } are complex involutory distributions
 F of dim $n$ such that 
 
 (1) $\omega(F_p F_p ) = 0$ at all $p \in X$ and 
 
 (2) $F + \bar{F}$ is an involutory distribution of constant dimension on $M$. Polarized sections  are sections satisfying $\nabla_{\bar{X}} s =0$ where $X$ lies in $F$.

The following  method of quantization  developed by Kostant and Souriau,  assigns to 
functions $f \in C^{\infty}(M)$, an operator, $\hat{f} = \nabla_{X_f} + 2 \pi i f$ acting on
 the Hilbert space of prequantization  described above. 
 Here $\nabla$ is the covariant derivative of  the  connection $\theta$ of the bundle  and $X_f$ is the Hamiltonian vector field with 
respect to $f$.

  Functions mapping polarized sections to polarized sections are the ones which survive quantization. A sufficient condition for a function $f$
  to do this is that its Hamiltonian vector field $X_f$ involutes with  the  polarization. This can be seen from the following 
  \begin{equation}\label{quant}
    \nabla_{\bar{X }} \hat {f} s = \hat{f}(\nabla_{\bar{X}} s ) + \nabla_{[\bar{X},X_f]} s
   \end{equation}
  here $X$ lies in the polarization and $s$ a polarized section, ~\cite{W}.
   
   \subsection{ Polarization of the Complexified Lie Algebra and that on the Manifold}
  We have discussed two polarizations. One on the complexified lie algebra and the the other on the manifold through involutory distributions.
  When there is a Hamiltonian group action on the manifiold $M$ of a group $G$ the fundamental vector fields are Hamiltonian and thus yielding
   functions. When $M$ is a Hamiltonian-$G$ space the assignment is a lie agebra  homomorphism.
    In case of codajoint orbits  lie algebra polarisation( described above) induces a polarisation in the manifold through fundamental vector fields. 
    In such a case 
    whole lie algebra 
    $\mathfrak{g}_C$ maps to polarized functions thus giving  representation of the lie algebra to the space of polarized sections, ~\cite{K2}.
    
     Since the Toda space $T$ is a  coadjoint orbit from what proved above there is a Hamiltonian group action of $G_1$  the group of invertible
     lower triangular matrices with non-negative diagonal entries of  determinant 1. It can be checked that restricting to $G_1$ does 
     not reduce the orbit since  $G=DG_1$  where $D$ are scalar matrices 
     matrices which lies in stabilizer of the coadjoint action on the Toda space.
      Since our  polarization comes from a lie algebra polarization  $\mathfrak{h}$ (lower triangular matrix with diagonal zero) which satisfies certain criterions (section $4$), 
the  polarized sections will be  a module      of a representation of $\mathfrak{g}_{\mathbb{C}}$. Exponentiating this action we get a 
representation of the group,  \cite{K2} . (More details in section $6$). 
  
     \section{Polarized Sections}
        
     When  the manifold $M$ is a coadjoint orbit the Hamiltonian functions due to fundamental vector fields can be evaluated easily. 
        The Hamiltonian function for a polarized fundamental vector fields is given by
        \begin{equation}\label{Hamil}
         H(l^{*})= Trace(l l^{*})
        \end{equation}
    where $l \in \mathfrak{h}$, as discussed above . Since in our case  diagonal elements of $l \in \mathfrak{h}$ are zero and restricting to real matrices we have
    
    \begin{equation}\label{Hamil_1}
     H( l^{*})= {\sum_{i=1}}^{n-1} l_{{1+1}{i}} l^{*}_{{i}{i+1}}
    \end{equation}
 going to $a_i$ and $b_i$ coordinate we get
 \begin{equation}\label{Hamil_2}
     H( l^{*})= {\sum_{i=1}}^{n-1} l_{1+1 i} a_i.
    \end{equation}
      
      Recall that the symplectic form is 
      \begin{center}
      $\omega = \sum _{i=1}^{n-1} \frac{1}{a_i} d(a_i) \wedge d(\sum_{j=1}^{i}b_{j})$.
      \end{center}
      Substituting
   \begin{equation}\label{mom}
    p_i={\sum_{j=1}}^{i}b_{j}
   \end{equation}
     we have
     $\omega = \sum _{i=1}^{n-1} \frac{1}{a_i} d(a_i) \wedge d(p_i)$
         So from above
         \begin{equation}\label{Vec}
         X_{H}={\sum_{i=1}}^{n-1} a_i l_{i+1 i} \partial{p_i}.
        \end{equation}
        
        So the polarisation is given by the fields $\partial_{p_i} $.
         Now choosing the symplectic potential (or the connection form) to be 
         \begin{center}
         $ \theta = {\sum_{i=1}}^{n-1}\frac{d(a_i)}{a_i}p_i$ 
         \end{center}
         the polarized sections will be solutions to the 
          following PDE 
          \begin{center}
          $(d(s)(X))  + i \theta(X)  s) =0$ 
          \end{center}
          where the section $s$ is just a smooth function due to triviality of the bundle and  
         \begin{center}
         $X = \sum_i  y_i(a, p) \partial_{p_i}$,
         \end{center}
         where $y_i$ is arbitrary and 
         \begin{center}
         $(a, p) = (a_1, a_2,...a_{n-1}, p_1,..,p_{n-1})$. 
         \end{center}
         
         Since 
         \begin{center}
         $\theta(X) =0$, $s=s(a, p)$ 
         \end{center}
         solves the equation
          \begin{equation}\label{pde}
           {\sum_{i=1}}^{n-1} y_i(a,p)\partial_{p_i} s(a,p) =0.
          \end{equation}
  This is because $X = \overline{X}$ in this case. Since  $\omega$ is exact and the Toda space is contractible, we have a trivial bundle and 
  the sections are functions and from the above,  the polarised 
  functions depend only on $a_i$.

 \section{ Unitary Representation}
 
 \subsection{ Induced Representations}
 
 Let us recall the definition of induced representation. 
Suppose $G$ is a topological group and $H$ is a closed subgroup of $G$. Suppose $\pi$ is a representation of $H$ over the vector space $V$. Then $G$ acts 
on the product $G \times V$ as follows:
\begin{center}
$g^{\prime} .(g,v)=(g^{\prime} g,v)$
\end{center}

where $g$ and $g^{\prime}$ are elements of $G$ and $v \in V$.

Define on $G \times  V$ the equivalence relation

\begin{center}
$(g,v) \sim (gh,\pi(h^{-1})(v)) \text{ for all }h\in H$.
\end{center}

Note that this equivalence relation is invariant under the action of $G$. Consequently, $G$ acts on $(G \times V)/ \sim $ . The latter is a vector bundle over the 
quotient space $G/H$ with $H$ as the structure group and $V$ as the fiber. Let $W$ be the space of sections of this vector bundle. This is the vector space underlying 
the induced representation $Ind^G_H \pi$. The group $G$ acts on $W$ as follows:

\begin{center}
$(g.\phi)(\overline{k})=g.\left(\phi(g^{-1}\overline{k})\right) \ \text{ for } g\in G,\ \phi\in W,\ \overline{k}\in G/H.$
\end{center}

 \subsection{Unitary Representation of $G$ on the Space Of Sections}
 
 We refer to the theory developed in ~\cite{K1}, ~\cite{K2}.
  When $M$ is a coadjoint orbit of a group $G$  and $g \in M$ from theorem $5$ in \cite{K2} we have a bijection between $ {\mathfrak{L}}_c(M , \omega)$ and
  ${\mathfrak{L}}_{g}$, (notation as in ~\cite{K2}). Here ${\mathfrak{L}}_{g}$ is the set of characters from  stabilizer group $G_{g}$ to the unit circle at  the fiber at
  the point $g$ of the corresponding element $\mathfrak{l} \in {\mathfrak{L}}_c(M , \omega)$. Calling the character ${\eta}^{\mathfrak{l}}$ we
  get a induced representation of the whole group on  the space of sections, ~\cite{K2}.
  
  \subsubsection{ Prequantization:}
  
  For prequantization, in the definition of induced representation we take 
  
  $G = G_1$, $H = G_T$, the stabilizer subgroup of $T$ in the dual of the Lie algebra. $T $ is a point on the coadjoint orbit $ G / G_T$. The representation 
 of $H$ is given by the character $\eta^{\mathfrak{l}}$.  As mentioned in theorem $5$ in ~\cite{K2}, the unitary representations of $G$ defined by exponentiation 
 of $\gamma^{\mathfrak{l}}$ o $\lambda$ is $\text{ind}_{G} \eta^{\mathfrak{l}}$. Here, recall,  $\lambda(x) (T) =  < T, x>$, which, in our case,  is $Tr(xT)$ where $x$ is in the 
 Lie algebra and $T$ in its dual, both are given by matrices. Also 
 $\gamma^{{\mathfrak{l}}}(\phi) s = ( \nabla_{\zeta_{\phi}} + 2 \pi i \phi) s $, $s$ is a section of the prequantum line bundle 
 and $\phi$ is a smooth function on the coadjoint orbit.

   \subsubsection{Quantization:}
   
   We refer to the last paragraph of ~\cite{K2}.
   We chose our polarization ${\mathfrak{h}}$ which satisfied the criterion of Kostant's polarization (see section $4$).
   Then, the representation $\text{ind}(\eta^{{\mathfrak{l}}}, {\mathfrak{h}})$ of $G$ corresponds to the unitary representation obtained from exponentiation of $ \gamma_F^{\mathfrak{l}} ({\mathfrak{h}}) $ o $\lambda$. (Notation as in ~\cite{K2}). 
   
   \subsubsection{Unitary Representation:}
   
   Recalling the space of polarized sections of the Toda system  are functions depending on variables $a_i$. The 
   action of the induced representation on the inner product is given
   by
    \begin{equation}\label{integral}
 \int_{T} {\lvert g^{*}s(a_1 \ldots a_{n-1})\rvert}^{2} {\lvert h(x)\rvert}^{2} d(V)=\int_{T} {\lvert s(a_1 \ldots a_{n-1})\rvert}^{2} {\lvert h(gx)\rvert}^{2} g^{*}(d(V))
   \end{equation}
   
 where $g \in G_1$ and  $T$ and $d(V)$ is  the orbit and its volume form respectively.  In our case $T$ is the Toda space and $d(V)$ is its volume form induced by the
 symplectic form and $h$ the  Hermitian structure compatible
   with the symplectic potential or the connection form. From  \cite{K1}, ~\cite{K2}, a Hermitian structure  $h$  is $\alpha$ invariant if
\begin{equation}\label{connection}
 2 \pi i(\alpha(s_0)-\bar{\alpha(s_0)})=d(log{\lvert h \rvert}^{2})
\end{equation}
where $s_0$ is the section defining the trivialization.
 In our  trivialization, $\alpha(s_0) = \theta$  where $\theta$  is real and hence  $\lvert h \rvert$ is constant.
   
  If we take $d(V)$ to be ${\omega}^{n-1}$  we have a   $g$ invariant volume since $\omega$ is $g$ invariant. But then square integrals of 
  functions
   depending on $a_i s$ will blow up since $b_i s$   vary over a unbounded domain. This will make the inner
  product to blow up. So we multiply by a normalizing factor
  \begin{equation}\label{vol}
   d(V)= {e}^{-\sum_{i=1}^{n-1} {p_{i}}^{2}} (\omega)^{n-1}
  \end{equation}
 where $p_i= \sum_{j=1}^{i} b_j$
  The action of the group $G_{1}$ on the Toda space  can be inferred from the section  \ref{Toda_1}.
        The variables $a_i$ and $b_i$ transform in following way when an element $g$ acts on it.
        \begin{equation}\label{act_1}
         g(a_i) = \frac{L_{ii}}{L_{i+1 i+1}} a_i.
        \end{equation}
          For $i>1$
        \begin{equation}\label{act_2}
         g(b_i) =  b_i +\frac{L_{i i-1}}{ L_{ii}} a_{i-1} - \frac{ L_{i+1 i}}{L_{i+1 i+1}}a_{i}.
        \end{equation}
        
        For $i=1$
        
         \begin{equation}\label{act_3}
         g(b_1) =  b_1  - \frac{ L_{2 1}}{L_{22}}a_{1}
        \end{equation}
   where $g_{ij} = L_{ij}$ .
    Substituting $p_i= \sum_{j=1}^{i}  b_j$
     we have
      \begin{equation}\label{act_4}
         g(p_i) =  p_i  - \frac{ L_{i+1 i}}{L_{i+1 i+1}}a_{i}.
        \end{equation}
        
        \begin{equation}\label{act_5}
         g^{*}(\d(V))= {e}^{-{\sum_{i=1}}^{n-1}{(p_{i}-\frac{ L_{i+1 i}}{L_{i+1 i+1}}a_{i})}^{2}} {\omega}^{n-1}.
        \end{equation}

         So the inner product $(gs, gs)$ is given by the following expression:
         
         \begin{equation}\label{act_6}
\int_0^{\infty} \ldots \int_0^{\infty}  \int_{-\infty}^{\infty}\ldots \int_{-\infty}^{\infty} {\lvert s(a_1 \ldots a_{n-1})\rvert}^{2} {e}^{-{\sum_{i=1}}^{n-1} {(p_{i}-\frac{ L_{i+1 i}}{L_{i+1 i+1}}a_{i})}^{2}}{\omega}^{n-1}.       
         \end{equation}
         We first perform the $p_i$-integrals.
 Since the $p_i$ coordinates range from  $(-\infty,\infty)$, making a substitution
 $y_i=p_i  -  \frac{ L_{i+1 i}}{L_{i+1 i+1}}a_{i}$ we get
 
 \begin{equation}\label{unit}
  (gs,gs)=(s,s).
 \end{equation}
 
 This  means $g$ is a  unitary representation. From last page  of \cite{K2} we have this representation is the exponential of the lie algebra 
 representation defined through connections and Hamiltonian functions described above.
   Since the $p_i$ parts of the integral evaluates to constant the inner product is
 
 \begin{equation}\label{inner}
  (s,s)=c \int_0^{\infty} \ldots \int_0^{\infty} {\lvert s(a_1 \ldots a_{n-1}) \rvert}^{2} \frac{d(a_1) \ldots d(a_{n-1})}{a_1 \ldots a_{n-1}}.
 \end{equation}
Thus the  Hilbert space of quantization, $\mathfrak{H}$,  is the completion of smooth functions of $a_i$'s  are such that the above integral  is finite.
\begin{remark}\label{ext}
 Since lower triangular matrices $L$ with positive diagonal entries
is given by $D G_1$ where $G_1$ is the group of our representation and $D$  scalar and since
$D$ has a trivial action on coadjoint orbit our representations can be extended to the whole group
$L$.
\end{remark}

\section{Coherent States}

\subsection{Rawnsley Coherent States}
In ~\cite{R} Rawnsley defined coherent states corresponding to geometric quantization of compact Kahler manifolds.
We repeat the construction for our case (namely, $\mathfrak{L}^2$-completion of smooth functions). 

In our case the pre-Hilbert space consists of sections which are smooth functions of $a_i$ such that the integral in the above section is finite.
 After completing the space w.r.t. the  ${\mathfrak{L}}^2$ norm we get a Hilbert space $\mathfrak{H}$. By substituting $x_i = log(a_i)$ we find that the  Hilbert space lies in 
${\mathfrak{L}}^{2}({\mathbb{R}}^{n-1})$.

Consider the 
associated $C^{*}$ principal bundle $L_0 \rightarrow M $ of $L$ . Take  an element $ q \in L_0$ 
  define a function $\Delta_q$
  on the pre-Hilbert of  space sections described above   such that 
  $\Delta_q(s).q = s(\pi q )$. 
   
   Since we have a trivial bundle the smooth sections  are smooth functions and $\Delta_q(s)$ is a constant times the  
   evaluation map of $s$ at $\pi(q)$.
   \begin{lemma}\label{eval}
    The evaluation map at a point on square integrable  smooth functions in ${\mathfrak{L}}^{2}$ norm is continuous.
   \end{lemma}
   \begin{proof}
     Suppose $s_n$ converge to $s$   and $s_n$  and $s$ smooth then $s_n$ converge point wise. Suppose not implies 
     there exist  $x$ such that $s_n(x)$ does not converge to $s(x)$. This 
     implies there exist a sub sequence  $s_{n_k}$ and  an $\epsilon$ such that $\lvert s(x)-s_{n_k}(x) \rvert > \epsilon$. Now 
     $ s^{-1}(x-\frac{\epsilon}{2}, x+\frac{\epsilon}{2})$ is
     an open set $U$ of positive measure. So from the above ${\lvert s-s_{n_{k}} \rvert}_2$ is  greater than $\epsilon {m(U)}^{\frac{1}{2}}$ where $m(U)$ is the Lebesgue measure of $U$ and 
     ${\lvert \ldots \rvert}_2$ is
     the ${\mathfrak{L}}^{2}$ norm. This contradicts ${\mathfrak{L}}^{2}$ convergence of $s_n$ to $s$. 
     
     \end{proof}

    \begin{lemma}\label{Ext_1}
    Any continuous linear functional $f$ on a pre-Hilbert $V$ space extends to a continuous linear functional $\tilde{f}$ on  a 
    Hilbert space $\overline{V}$. 
   \end{lemma}
 \begin{proof}
 This can be found in any standard functional analysis text book  but we give a proof for the reader's convenience. Suppose we 
 have a sequence $s_n$  in $V$
 which converges to $s$ in $\overline{V}$ then $f(s_n)$ converges to $\tilde{f}(s)$ by definition. Now if we have a sequence $\tilde{s}_n$ which 
 converges to $s$ where $\tilde{s}_n$ in $\overline{V}$, then we can find $s_n \in V$ such that ${\lvert s_n-\tilde{s}_n \rvert}_{\overline{V}}< \frac{1}{n}$ and 
 $\lvert f(s_n)-\tilde{f}(\tilde{s}_n)\rvert< \frac{1}{n}$. Since $\tilde{s}_n$ converge to $s$ then  $s_n$ converge to $s$ so by 
 definition $f(s_n)$ converges to $\tilde{f}(s)$. Since we have
 chosen $s_n$ such that  $\tilde{f}(\tilde{s}_n)$ and $f(s_n)$ converge to the same point we have $\tilde{f}(\tilde{s}_n)$ converges to $\tilde{f}(s)$.
 \end{proof}
 \begin{corollary}\label{Ext_2}
  The functional  $\Delta_q$ extend to continuous linear functional $\tilde{\Delta_q}$ on the completion.
 \end{corollary}

  So by Riesz representation theorem
  there is a section $e_q$ such that $\tilde{\Delta_q}(s)q= (e_q,s)q$, where $( , )$ is the inner product on the Hilbert 
  space. For $q$ and $q_1$ in the same fiber $e_q$ and $e_{q_1}$  differ by a multiplication
  by a constant. Thus we get a map from the base to the projectivization of the Hilbert space. For every $x$ in $M$ the projectivized state it
  maps to is what  is called a  $Rawnsley$ coherent state.
 \begin{remark}\label{eval}

 Recall that  the $\Delta_q$'s are  evaluation 
 maps on the space of  smooth square integrable functions (where $q$ is the element $1$ on the fiber above $\pi(q) $ corresponding  to the  global trivialization). Suppose $s_n$ is a sequence of smooth sections (in our case functions) converging in  ${\mathfrak{L}}^{2}$ norm to $s$, then
 $s_n$ will converge pointwise to a function $s^{'}$.  Because of the continuous extension of  $\Delta_q$  (which 
 is just the evaluation for the $q$ we have 
 chosen) we have   $s^{'}(\pi(q))= \tilde{\Delta_q}(s)$. Now since ${\mathfrak{L}}^{2}$ convergence implies pointwise convergence of 
 a subsequence outside a set  of zero measure, $s_n$ converge to $s^{'}$ in ${\mathfrak{L}}^{2}$. 
  
   Though ${\mathfrak{L}}^{2}$  space are not evaluative 
   we can choose  the  function $s^{'}$  as a representative, i.e.  the 
   function $s^{'}(\pi(q))= \tilde{\Delta_q}(s)$ where $q$ is again the element $1$  on the fiber above $\pi(q) $  with respect to the  global trivialization, 
   $\tilde{\Delta_q}$ is the extension of $ \Delta_q$,  $s$ is 
   an element of ${\mathfrak{L}}^{2}$ lying in completion and $s^{'}$ a representative of $s$ described as  
  above. Hence taking these 
   special $s^{'}$ we get an evaluation of elements of the completion. For defining coherent states, the particular choice of $q$ does not matter, since any two choices of $q$ will differ by multiplication by a constant. 
  \end{remark} 
 
 Since we have a trivial bundle the  smooth sections are functions and from above 
 the Hilbert space $\mathfrak{H}$ also consists of evaluative functions 
   so that the inner product with  Rawnsley coherent  vectors for the special $q$ at $x$ is just the evaluation of the function
   at $x$. Now take an orthonormal basis of $\mathfrak{H}$ ,${\phi_{i}}$  with $i \in I$ for some index set $I$. Then the dot product of 
   $\phi_{i}$ with a function
      $f_x$ (denoting the coherent state which evaluates at $x$)  is $\phi_i (x)$. So the coherent state at $x$  is represented by the vector 
      be $f_x = \sum_{i \in I} \phi_i(x) \phi_i$.

  

\subsection{Coherent States from Group Actions and the Relation to Rawnsley Coherent States}

Given the unitary $G$ action on the Hilbert space of sections, states of the form $G \psi_0$, where  $\psi_0$ is the ``vacuum'' state, are  also called coherent states.

  Consider a generic section $s$ of the bundle $L$ on the  manifold (coadjoint orbit in our case) which is  nonzero outside a set of zero measure.
  Let the set on which it is nonzero be called $U$. Let $x \in U$  , then $s(x)\in L_0$ and so we can  have $e_{s(x)}$. So from above we have
  \begin{equation}\label{coher_1}
  (e_{s(x)},s_1) =\frac{s_1(x)}{s(x)}.
  \end{equation}
  For any smooth section $s_1$ in the Hilbert space.
   Now take $e_{s(x)}$ and let $g \in  G$ act on it. For $g e_{s(x)}$ to be a   Rawnsley's  coherent  vector a good guess for its value will be $c e_{gs(x)}$.
    So we have
    \begin{equation}\label{coher_2}
  (g e_{s(x)},s_1) =c\frac{s_1(gx)}{s(gx)}
  \end{equation}
  but by unitarity of the action
  \begin{equation}\label{coher_3}
    (g e_{s(x)},s_1)=(e_{s(x)}, g^{-1}s_1)=\frac{g^{-1}s_1(x)}{s(x)}.
    \end{equation}
    
    So by the above equations, for the two notions to agree,  we should have
   \begin{equation}\label{coher_4}
   c\frac{s_1(gx)}{s(gx)}=\frac{g^{-1}s_1(x)}{s(x)}.
   \end{equation}
  
  In general the coherent states obtained from group action may not be the same as Rawnsley coherent states.
  But there are cases when they are the same. 
  
  Let $G$ be the group of translations acting on ${\mathbb{R}}^{2n}$ with standard symplectic form. Further let  $g \in G $ act  by translation by  $w$, let  $x=0$
  and 
  let $\frac{s_1}{s}$ be represented  by the  function $\Phi(z)$.
  Then it can be seen  the left hand term of the above equation with $c=1$ is $\Phi(w)$ and right hand side is $\Phi((z-w)=0)$  which is $\Phi(w)$.
  So the two kinds of coherent states are same in case of the above group action.
  
  We have the following proposition for geometric quantization of coadjoint orbits.
  
\begin{proposition}\label{coher_5} For Kostant-Souriau quantization of coadjoint orbits, the two notions of coherent states are the same. 
\end{proposition}

\begin{proof}
In this case  the representation module is the space of sections and the representation is the 
 induced representation of a character   from stabilizer of an element of the orbit to the fiber of the orbit.  In this case case the two notions of coherent states are the same as can be seen below.

 By definition of induced representation, (see section $6.1$),   $g^{-1}.\phi(\overline{k})=g^{-1}.(g \overline{k},v')=(\overline{k},v')$ where $v'$ is the representative of $\phi(g \overline{k})$ at $g \overline{k}$ when the section is pulled back to $G \times V$. 
   Now let the element $x$ be as described above and let $L_x$ is the fiber on which the character acts  and let $s_1=\phi$ and $\overline{k}=x$. So 
   $v'=\frac{s_1(gx)}{s(gx)} c^{\prime}$ where $c^{\prime}$ is the representative of  $s(gx)$ at $gx$ and $g^{-1}s_1(x)=v' \frac{s(x)}{c_1}=\frac{s_1(gx)}{s(gx)}\frac{c^{\prime}}{c_1} s(x) $ where
   $c_1$ is the representative of $s(x)$ at $x$.
Thus, $\frac{g^{-1} s_1(x)}{s(x)} = \frac{s_1(gx)}{s(gx)}\frac{c^{\prime}}{c_1}$.
In the local trivialization defined by $s(x)$,  one can always take $c_1 = c^{\prime} =1$. Thus in equation $(7.4)$ $c=1$.
\end{proof}

\section{ Finite Dimensional Representations}
 We can have finite representations of the group $G_1$
 \begin{equation}\label{fin}
  s(a_1 \ldots a_{n-1})=(a_1 \ldots a_{n-1}) P_m(a_1 \ldots a_{n-1}) 
 \end{equation}
where $P_m$  are polynomial in $a_i s$ of degree $m$. It can be checked $g(s(x)) =s(g^{-1} x)$ will keep the above space
of sections invariant.
But they are not square integrable. When we normalize by ${e}^{-\sum_{i=1}^{n-1} a_i^2} $ we may loose unitarity. So we  get finite 
representations which are not necessarily unitary. 

\vspace{.2in}

 \begin{remark}\label{nor}  Normalization of volume is equivalent to changing the Hermitian structure

Suppose a non-zero section $s$ locally  trivializes  a bundle $(L,\alpha)$ where $\alpha$ is a connection in the associated line bundle $L_{0}$. A Hermitian structure compatible
with the connection is an Hermitian metric $< >$ such that
\begin{center}
 $\zeta(<s,s>)= <\nabla_{\zeta}s,s> +<s,\nabla_{\zeta}s>$
\end{center}
and
\begin{center}
$ <\nabla_{\zeta}s,s> +<s,\nabla_{\zeta}s>=2\pi i<\alpha(\zeta) s, s>  - 2\pi i< s, \alpha(\zeta) s >$ 
\end{center}

where $s$ is a section which locally trivializes the bundle.

From this we get  ( see proposition 1.9.1 Kostant ~\cite{K1})
\begin{center}
 $d<s,s>= 2 \pi i (\alpha(s) - \overline{(\alpha(s))})$.
\end{center}

In the case of  a trivial bundle  with $s_0$ the section which gives the trivialization, if the connection $\alpha(s_0) $ is real, then we have  
$d<s_0,s_0>=0$. 
 
 Thus  $<s_0,s_0>=c$, where $c$ is a constant. Let us take $c=1$.
 
 Taking another section $s= \frac{s}{s_0} s_0$ we have
 \begin{center}
 $<s,s>={\lvert \frac{s}{s_0} \rvert}^{2} $.  
 \end{center}
 So in this setting we get a  unitary character ${\eta}^{l}$ and taking the volume form to be ${\omega}^{n-1}$ (which is invariant under the group action),
 we get  a unitary representation on the space of  square integrable sections . 
 
 In our case the polarized  sections are not  square integrable with $\omega^{n-1}$ as the volume form, but we do have a representation by virtue of the 
 fact they are sections and because of the polarization chosen. When we normalize the volume, we  change the Hermitian structure. However $\eta^{l}$ may not be  a compatible 
 unitary character in this Hermitian structure so we  may not get unitary representation in general.
 
 But in the case where we normalized by ${e}^{-\sum_{i=1}^{n-1} p_i^2} $, we do get unitary representation.
 \end{remark}
 \section{ The Quantum Hamiltonian of the Toda System}
 By substituting $q_i= log(a_i)$, $i = 1,...,n-1$ and $p_i= \sum_{j=1}^{i} b_j$, $i =1,...,n-1$ the Hamiltonian described in introduction is given by
 \begin{equation}\label{ham1}
  H= 2{\sum_{i=1}}^{n-1} (p_i-p_{i-1})^{2} +4{\sum_{i=1}}^{n} {e}^{2 q_i}
 \end{equation}
 with $p_0=0$.
 Now quantizing $p_i$ and $q_i$ ( using the formula $\hat{f}= \nabla_{X_f} + 2 \pi i f$ for $f =p_i$ and $q_i$ respectively), we 
 get $\hat{p_i}= -i \partial{q_i}$ and $\hat{q_i}=q_i$. Substituting in the expression for the Hamiltonian $H$, we get 
 
 \begin{equation}\label{ham2}
  \hat{H}= -2 \partial{q_1}^{2} -2{\sum_{i=2}}^{n-1} (\partial{q_i} - \partial{q_{i-1}})^{2}   +4{\sum_{i=1}}^{n-1} {e}^{2q_i}.
 \end{equation}

Let $q_i =z_{i} - z_{i+1}$ for $i<n-1$ and $q_{n-1}=z_{n-1}$.

Then 
\begin{center}
$\partial_{z_{i}} = \partial{q_{i}} - \partial{q_{i-1}}$ 
\end{center}

for $n-1 \geq i>1$ and 
\begin{center}
$\partial_{z_{1}} = \partial{q_{1}},$ 
\end{center}
 
 \begin{center}
$z_{i+1}-z_1 = -\sum_{j=1}^i q_j $  \rm{\;for\;}  $n-1>i>0$ .  
 \end{center}
 Further 
 \begin{center}
 \begin{equation}
 \hat{H} = -2 \sum_{i=1}^{n-1} (\partial_{z_i})^2 + 4 (\sum_{i=1}^{n-2} {e}^{2(z_i - z_{i+1})} + {e}^{2z_{n-1}}). 
 \end{equation}
 \end{center}

As in ~\cite{KL}, ~\cite{Se}, the eigenfunctions can possibly be related to Whittacker functions. Also the spectrum needs to be found. This is work in progress.

\section{ Acknowledgements} Rukmini Dey would like to thank Professor Dileep Jatkar, Professor Adimurthy, Professor Samir K Paul, Professor Leon Takhtajan  and Soumyadeep Chaudhuri for the useful discussions.
Saibal Ganguli would like to thank ICTS-TIFR Bangalore for hospitality while the work was carried out.

  Rukmini Dey \\
  ICTS-TIFR,\\
  Survey 151, Shivakote,\\
  Hesaraghatta Hobli, \\
  Bangalore 560089\\
  India\\
  {\it E-mail address:} {\tt rukmini@icts.res.in}\\[0.3cm]
  Saibal Ganguli
  HRI, \\
  Chhatnag Road Jhusi\\
  Allahabad 211019\\
  India\\
  {\it E-mail address}:{\tt saibalgan@gmail.com}
  
  \label{last}

\end{document}